\documentclass[reqno,a4paper]{amsart}

\usepackage{amsmath}
\usepackage{amsthm}
\usepackage{txfonts}
\usepackage{amsfonts}
\usepackage{amssymb}
\usepackage{graphicx}
\usepackage{subfig}
\newtheorem{theorem}{Theorem}[section]
\newtheorem{exttheorem}{Theorem}

\theoremstyle{plain}
\newtheorem{claim}{Claim}
\newtheorem*{claim*}{Claim}
\newtheorem{corollary}[theorem]{Corollary}

\newtheorem{lemma}[theorem]{Lemma}
\newtheorem*{question}{Question}
\newtheorem{proposition}[theorem]{Proposition}
\newtheorem{remark}[theorem]{Remark}

\newcommand\R{{\mathbb R}}

\newcommand\N{{\mathbb N}}
\newcommand\C{{\mathbb C}}
\newcommand\D{{\mathbb D}}
\newcommand\RS{{\hat {\mathbb C}}}
\DeclareMathOperator*{\inter}{int}

\begin{document}
\title[Non-self maps of regions in the plane]{Fixed Point Theorem for Non-Self Maps of Regions in the Plane}
\author{Georg Ostrovski}
\address{Mathematics Department, University of Warwick, CV4 7AL Coventry, United Kingdom}
\email{g.ostrovski@gmail.com}

\begin{abstract}
Let $X \subset \R^2$ be a compact, simply connected, locally connected set, and let 
$f \colon X \to Y \subset \R^2$ be a homeomorphism isotopic to the identity on $X$. 
Generalizing Brouwer's plane translation theorem for self-maps of the plane, 
we prove that $f$ has no recurrent (in particular, no periodic) points, if it has no fixed points. 
\end{abstract}

\maketitle

\section{Introduction and Statement of the Result}

Since Brouwer's proof of his plane translation theorem \cite{Brouwer1912}, 
many alternative proofs of the theorem and its key ingredient, the translation arc lemma, have been given 
(for several more recent ones, see Brown \cite{Brown1984}, 
Fathi \cite{Fathi1987}, Franks \cite{Franks2008}). The following is a 
concise formulation of its main statement.

\begin{exttheorem}[Barge and Franks \cite{Barge1993}]\label{thm:franks}
Suppose $f\colon \R^2\to \R^2$ is an orientation-preserving homeomorphism of the plane.
If $f$ has a periodic point then it has a fixed point.
\end{exttheorem}

Slightly stronger versions assume a weaker form of recurrence, 
for example the existence of periodic disk chains 
(Barge and Franks \cite{Barge1993}), to obtain the existence of fixed points.

The point of this is paper is to show that an analogue of Theorem \ref{thm:franks}
also holds for a homeomorphism which is merely defined  on a compact
subset of a surface. This situation could arise when considering a restriction
of a self-map of a surface. Our main theorem is the following: 

\begin{theorem}\label{thm:main_thm}
 Let $X \subset \R^2$ be a compact, simply connected, locally connected subset of the real plane
and let $f \colon X \to Y \subset \R^2$ be a homeomorphism isotopic to the identity on $X$. 
Let $C$ be a connected component of $X \cap Y$. If $f$ has a periodic orbit in $C$, 
then $f$ also has a fixed point in $C$.  
\end{theorem}

In fact, as a corollary we will obtain a slightly stronger result:
\begin{corollary}\label{cor:main_cor}
 Let $X \subset \R^2$ and $f \colon X \to Y$ be as in Theorem \ref{thm:main_thm} 
and let $C$ be a connected component of $X \cap Y$. 
If $f$ has no fixed point in $C$, then the orbit of every point $x \in C$ eventually leaves $C$, 
i.e., there exists $n = n(x) \in \N$ such that $f^n(x) \notin C$. 

In particular, if $X \cap Y$ is connected and $f$ has no fixed points, then
the \textit{non-escaping set} of $f$ is empty:
\begin{equation*}
 \{x \in X \colon f^n(x) \in C ~\forall n \in \N \} = \emptyset.
\end{equation*}
\end{corollary}

To prove these results, in Section \ref{sec:jordan} we will first consider
the case of an orientation-preserving homeomorphism
$f\colon D \to E \subset \R^2$ of a Jordan domain $D$ into the plane. 
In this special case, the statement will be first proved for $D \cap E$ connected, 
by suitably extending $f$ to the real plane and applying Theorem \ref{thm:franks}.
We will then proceed to show that the connectedness assumption can be removed
if one formulates the result more precisely, 
taking into account the connected components of $D \cap E$ individually. 
Finally, in Section \ref{sec:gen_sets}, we will deduce the general case of 
Theorem \ref{thm:main_thm} by reducing the problem to the Jordan domain case. 
In Section \ref{sec:discussion} we will discuss the assumptions of our results
and questions about possible extensions.

\section{Non-self maps of Jordan domains}\label{sec:jordan}

A set $D \subset \R^2$ is a \textit{Jordan domain}, if it is a compact set with boundary $\partial D$ a
simple closed curve (\textit{Jordan curve}). By Schoenflies' theorem, 
Jordan domains are precisely the planar regions homeomorphic to the (closed) disk. 
In this paper we assume that all Jordan curves are endowed with the counter-clockwise orientation. 
For a Jordan curve $C$ and $x,y \in C$, we denote by $(x,y)_C$ (respectively $[x,y]_C$) the open (respectively closed) arc
in $C$ from $x$ to $y$ according to this orientation. 

For $X \subset \R^2$, we say that $x \in X$ is a \textit{fixed point} 
for the map $f \colon X \to \R^2$ if $f(x) = x$. 
We say that $x \in X$ is a \textit{periodic point} for $f$ 
if there exists $n \in \N$ such that $f^n(x) = x$. This of course requires
that the entire \textit{orbit} of $x$, $\mathcal{O}_f(x) = \{x = f^n(x), f(x), \ldots, f^{n-1}(x)\}$, 
is included in $X$. 

In analogy to Theorem \ref{thm:franks}, we will prove the following:

\begin{theorem}\label{thm:conn_jordan}
Let $D \subset \R^2$ be a Jordan domain and let $f \colon D \to E \subset \R^2$ 
be an orientation-preserving homeomorphism. Assume that $D \cap E$ is connected.
If $f$ has a periodic point, then $f$ also has a fixed point.  
\end{theorem}

\begin{remark}
All periodic or fixed points of $f \colon D \to E$ necessarily lie in $D \cap E$. 
\end{remark}

The strategy of our proof is to show that a homeomorphism $f \colon D \to E$ as in Theorem \ref{thm:conn_jordan} 
and without fixed points can be extended to an orientation-presering fixed point free homeomorphism 
$F \colon \R^2 \to \R^2$ (a \textit{Brouwer homeomorphism}). 
The result then follows from classical results such as Theorem \ref{thm:franks}.

Our first step is to consider the structure of the set $D \cap E$. 
It is known that each connected component of the intersection of two Jordan domains
is again a Jordan domain (see, e.g., Ker\'ekj\'art\'o \cite[p.87]{Kerekjarto1923}). 
We need the following more detailed statement (Bonino {\cite[Proposition 3.1]{Bonino2002}}, 
proved in Le Calvez and Yoccoz {\cite[Part 1]{LeCalvez1997}}).

\begin{proposition}\label{prop:jordan}
 Let $U, U'$ be two Jordan domains containing a point $o \in U \cap U'$
such that $U \nsubset U'$ and $U' \nsubset U$. 
Denote the connected component of $U \cap U'$ containing $o$ by $U \wedge U'$. 
\begin{enumerate}
 \item There is a partition
\begin{equation*}
 \partial (U \wedge U') = 
\left( \partial (U \wedge  U') \cap \partial U \cap \partial U' \right)
\cup \bigcup_{i\in I} \alpha_i \cup \bigcup_{j\in J} \beta_j~, ~\text{where}
\end{equation*}
\begin{itemize}
 \item $I,J$ are non-empty, at most countable sets, 
\item for every $i \in I$, $\alpha_i = (a_i,b_i)_{\partial U}$ 
is a connected component of $\partial U \cap U'$, 
\item for every $j \in J$, $\beta_j = (c_j,d_j)_{\partial U'}$ 
is a connected component of $\partial U' \cap U$. 
\end{itemize}
\item For $j \in J$, $U \wedge U'$ is contained in the Jordan domain 
bounded by $\beta_j \cup [d_j, c_j]_{\partial U}$. 
\item $\partial (U \wedge U')$ is homeomorphic to $\partial U$, hence is
itself a Jordan curve. 
\item Three points $a,b,c \in \partial (U \wedge U') \cap \partial U$
(respectively $\partial (U \wedge U') \cap \partial U'$) are met in this
order on $\partial U$ (respectively $\partial U'$) if and only if they are met in
the same order on $\partial (U \wedge U')$. 
\end{enumerate}
\end{proposition}

We consider a Jordan domain $D \subset \R^2$ and a homeomorphism 
$f \colon D \to E = f(D) \subset \R^2$. If $D \cap E$ is connected, it follows that 
both $D \cap E$ and $D \cup E$ are Jordan domains (see Fig.\ref{fig:prop_jordan}). 
\begin{figure}
 \centering
\includegraphics[width = 0.5\textwidth]{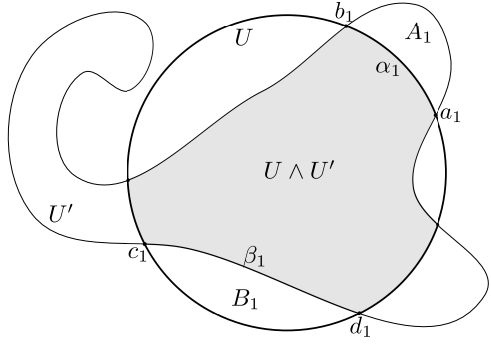}
\caption{(Proposition \ref{prop:jordan}) Partition of $\partial (U \wedge U')$ in the case when $U \cap U'$ is connected.}
\label{fig:prop_jordan}
\end{figure}
With Proposition \ref{prop:jordan} in mind, the proof of the following lemma is straightforward; 
it can be found in the appendix. 

\begin{lemma}\label{lem:part}
Let $D, E \subset \R^2$ be Jordan domains such that $D \nsubset E$, $E \nsubset D$ 
and $D \cap E$ is non-empty and connected. 
Then there exists a partition of $\R^2 \setminus \inter(D \cap E)$ into arcs, each of which connects a point on 
$\partial (D \cap E)$ to $\infty$ and intersects each of $\partial D$ and $\partial E$ in precisely one point.

Furthermore, there exist continuous functions 
$\lambda^D \colon \R^2 \setminus \inter(D) \to \R_{\geq 0}$ and $\lambda^E \colon \R^2 \setminus \inter(E) \to \R_{\geq 0}$, 
which are strictly monotonically increasing along the arcs of the partition, 
and such that $\lambda^D\vert_{\partial D} \equiv 0$ and $\lambda^E \vert_{\partial E} \equiv 0$.
\end{lemma}

\begin{remark}
 For a partition arc $\gamma$ connecting some point $p \in \partial D$ (or $\partial E$) to $\infty$, 
the function $\lambda^D$ ($\lambda^E$) can be viewed as providing a notion of arc length (assigning to each $x \in \gamma$ 
the length of the subarc $[p,x]\subset\gamma$). It is finite for every $x \in \gamma$, 
even if the usual (Euclidean) arc length $\ell([p,x])$ is not. 
\end{remark}

The partition of $\R^2 \setminus \inter(D \cap E)$ obtained this way allows us to prove our
key proposition, from which the main result of this section will follow as a corollary. 

\begin{proposition}\label{prop:extension}
Let $D \subset \R^2$ be a Jordan domain and $f \colon D \to E \subset \R^2$ an
orientation-preserving homeomorphism with no fixed points, and 
such that $D \cap E$ is non-empty and connected.
Then there exists a fixed point free orientation-preserving homeomorphism $F \colon \R^2 \to \R^2$ 
extending $f$ (i.e., $F\vert_D \equiv f$). 
\end{proposition}

\begin{proof}
Since $f$ has no fixed points, 
by Brouwer's Fixed Point Theorem we get that $E \nsubset D$ and $D \nsubset E$.
Therefore, we can apply Lemma \ref{lem:part} to obtain a partition $\mathcal{P}$ of $\R^2 \setminus \inter(D \cap E)$ into arcs, 
each connecting a point in $\partial (D \cap E)$ with $\infty$ and 
intersecting each of $\partial D$ and $\partial E$ in exactly one point.

For $x \in \R^2 \setminus \inter(D \cap E)$, 
we denote by $L^x \in \mathcal{P}$ the partition element containing the point $x$, and let 
$\pi^D(x) = L^x \cap \partial D$, $\pi^E(x) = L^x \cap \partial E$. From the contruction of 
$\mathcal{P}$ in the proof of Lemma \ref{lem:part} it is clear that $\pi^D\colon \R^2 \setminus \inter(D \cap E) \to \partial D$ and 
$\pi^E\colon \R^2 \setminus \inter(D \cap E) \to \partial E$ are continuous.

We now construct $F \colon \R^2 \to \R^2$ as an extension of $f$ in such a way, 
that each arc $L^x \setminus D$, $x \in \partial D$, is mapped into the arc $L^{f(x)}$. 
For $x \in D$ we define $F(x) = f(x)$, and for $x \notin D$ we define $F(x)$ to be the unique point 
$y \in L^{f(\pi^D(x))} \setminus E$ such that $\lambda^E(y) = \lambda^D(x)$ (see Fig.\ref{fig:proof_ext}). 
Then, by the construction of $\mathcal{P}$ and continuity of $\pi^D$ and $\pi^E$, $F$ is a continuous extension of $f$. 
\begin{figure}
  \centering
    \includegraphics[width=0.4\textwidth]{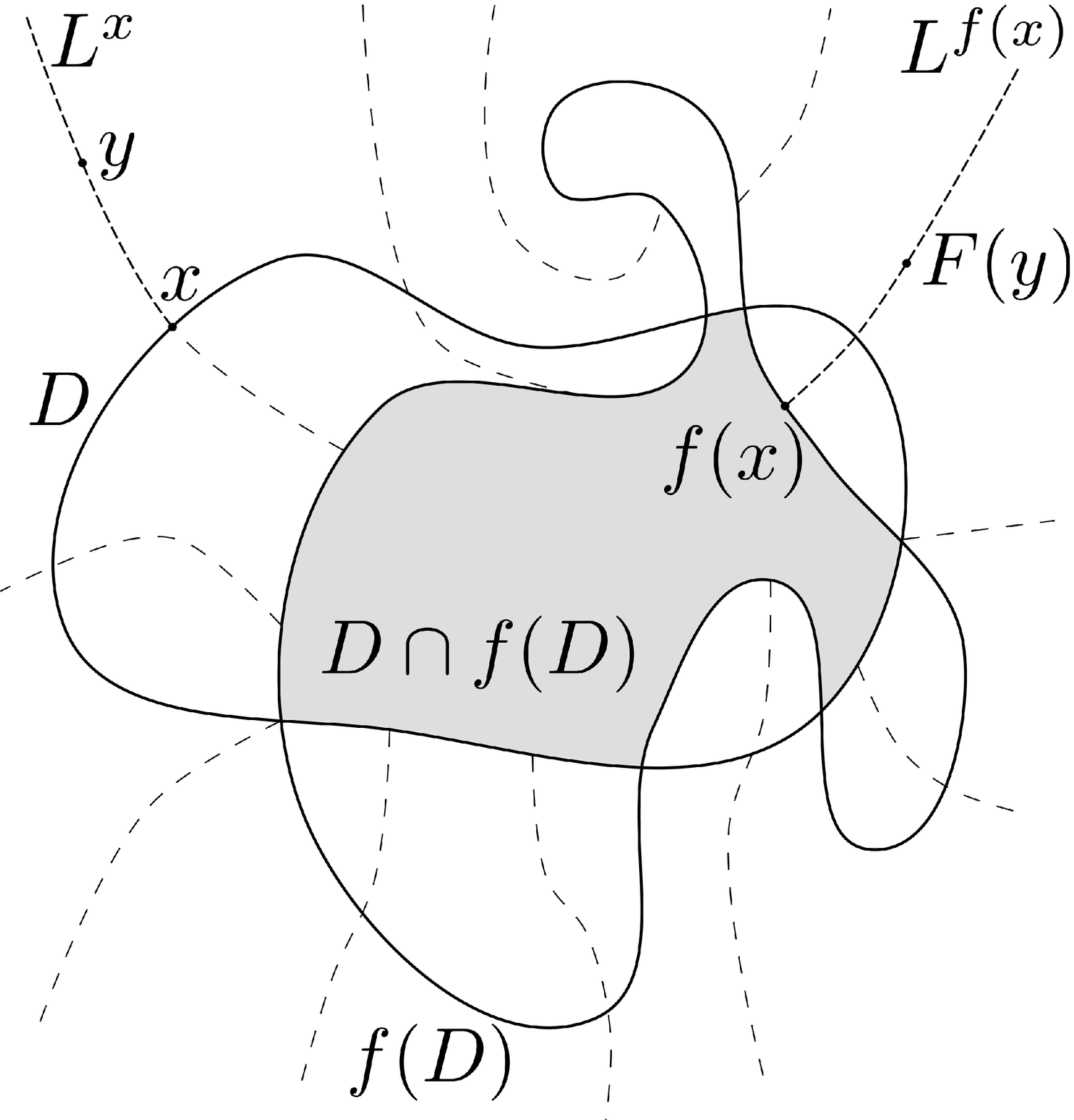}
  \caption{(Proof of Proposition \ref{prop:extension}) Extension $F \colon \R^2 \to \R^2$ of $f \colon D \to f(D)$ 
maps each arc $L^x$, $x \in \partial D$, into the arc $L^{f(x)}$.}
\label{fig:proof_ext}
\end{figure}

The map $F$ is an orientation-preserving homeomorphism of $\R^2$; 
its inverse can be obtained the same way by swapping the roles of $D, f$ and $E, f^{-1}$.
Moreover, $F$ has no fixed points in $D$ since $F\vert_D \equiv f$.
On the other hand, suppose $p \in \R^2 \setminus D$ with $F(p) = p$.  
Then $F(L^p \setminus D) \subset L^p$. 
Since $\lambda^D(p) = \lambda^E(F(p)) = \lambda^E(p)$, we get $\pi^D(p) = \pi^E(p)$, 
and therefore $f(\pi^D(p)) = \pi^D(p)$, a contradiction. Hence $F$ is fixed point free. 
\end{proof}

Combining Proposition \ref{prop:extension} and Theorem \ref{thm:franks}, 
we now obtain Theorem \ref{thm:conn_jordan}. 

\medskip

Next, we formulate a somewhat more general form of Theorem \ref{thm:conn_jordan}, 
by removing the restriction of $D \cap f(D)$ to be connected, and instead making an assertion
for each individual connected component of this set. This version of the result
turns out to be much more useful in its application to more general compact sets 
in Section \ref{sec:gen_sets}. 

\begin{theorem}\label{thm:jordan_general}
Let $D \subset \R^2$ be a Jordan domain and 
$f \colon D \to E \subset \R^2$ an orientation-preserving homeomorphism. 
Let $C$ be a connected component of $D \cap E$. 
If $f$ has a periodic orbit in $C$, 
then $f$ also has a fixed point in $C$.  
\end{theorem}
 
\begin{proof}
 We show that this more general setting can be reduced to the one in Theorem \ref{thm:conn_jordan}.
Let $\{ C_i \}$ be the collection of connected components of $\overline{ E \setminus C}$
and let $\gamma_i = C \cap C_i$, which is a closed subarc of $\partial D$ (a connected component of $C \cap \partial D$). 

Each $C_i$ is a Jordan domain bounded by the union of $\gamma_i$ and a closed subarc of $\partial E$. 
Let $\{\tilde C_i\}$ be another collection of disks $\tilde C_i \subset \overline{\R^2 \setminus D}$ such that $\tilde C_i \cap D = \gamma_i$ 
and the $\tilde C_i$ are pairwise disjoint. (Note that by Schoenflies' theorem we can assume without loss of generality $D$ 
to be the closed standard unit disk $\overline \D = \{(x,y) \in \R^2 \colon x^2+y^2 \leq 1 \}$. Since the $\gamma_i$ are pairwise disjoint, 
letting $\tilde C_i$ be the round halfdisk over $\partial \D$ 
centred at the midpoint of $\gamma_i$ gives a collection of pairwise disjoint disks as required, see Fig.\ref{fig:genjordanproof}.)
\begin{figure}
  \centering
  \includegraphics[width=0.45\textwidth]{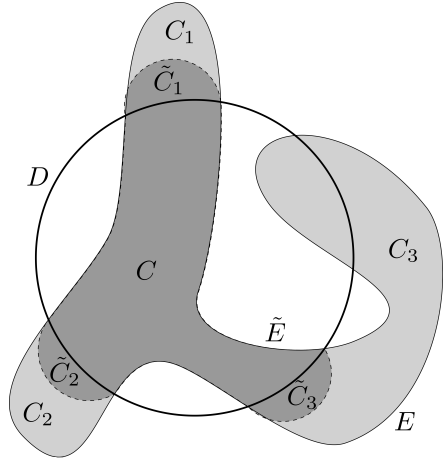}
\caption{(Proof of Theorem \ref{thm:jordan_general}) The Jordan domains $D$ (white), $E = \left(\bigcup_i C_i\right) \cup C$ (light grey)
and $\tilde E = \left(\bigcup_i \tilde C_i\right) \cup C$ (dark grey). The $\tilde C_i$ are pairwise disjoint, the homeomorphism 
$g \colon E \to \tilde E$ maps each $C_i$ homeomorphically to $\tilde C_i$ and fixes $C$ pointwise.
Moreover, $\tilde E \cap D = C$ is connected.}
\label{fig:genjordanproof}
\end{figure}

By Schoenflies' theorem, for each $i$ there exists a homeomorphism $g_i \colon C_i \to \tilde C_i$ and we can assume
that $g_i$ is the identity on $\gamma_i$.
Noting that $E = \left(\bigcup_i C_i\right) \cup C$, let $\tilde E = \left(\bigcup_i \tilde C_i\right) \cup C$ 
and define a homeomorphism $g \colon E \to \tilde E$ by
gluing together the identity on $C$ and the homeomorphisms $g_i$:
\begin{equation*}
 g(y) = \begin{cases} y &\mbox{if } y \in C, \\ g_i(y) &\mbox{if } y \in C_i. \end{cases}
\end{equation*}

Finally, define the homeomorphism $\tilde f = g \circ f \colon D \to \tilde E$.
By construction, $\tilde f$ and $f$ coincide on $f^{-1}(C)$, 
and if $f$ has a periodic orbit in $C$, so does $\tilde f$. 
Since $\tilde f(D) \cap D = \tilde E \cap D = C$ is connected, we can apply Theorem \ref{thm:conn_jordan} to get that $\tilde f$
has a fixed point in $C$. Hence $f$ has a fixed point in $C$. 
\end{proof}

We formulate a slightly stronger result as a corollary to the proof of Theorem \ref{thm:jordan_general}:

\begin{corollary}\label{cor:jordan_general_rec}
Let $f \colon D \to E$ be as in Theorem \ref{thm:jordan_general} and 
$C$ a connected component of $D \cap E$. 
If $f$ has no fixed point in $C$, then the \textit{non-escaping set} of $f$ in $C$ is empty:
\begin{equation*}
  \{x \in C \colon f^n(x) \in C ~\forall n \in \N \} = \emptyset.
\end{equation*}
\end{corollary}

\begin{proof}
By repeating verbatim the argument in the proof of Theorem \ref{thm:jordan_general}, 
we can construct a homeomorphism $\tilde f \colon D \to \tilde E$, which coincides with $f$ on $C$ 
and such that $\tilde E \cap D = C$. 

Observe that $\tilde f$ is fixed point free, so we can apply 
Proposition \ref{prop:extension} to get an extension of $\tilde f$ to a 
fixed point free orientation-preserving homeomorphism $F \colon \R^2 \to \R^2$. 
By \cite[Corollary 1.3]{Franks2008}, $F$ is a so-called \textit{free} homeomorphism,
i.e., for any topological disk $U \subset \R^2$ we have 
\begin{equation}\label{eq:freedisk}
F(U) \cap U = \emptyset \Rightarrow F^i(U) \cap F^j(U) = \emptyset ~ \text{ whenever }~ i \neq j.
\end{equation}

Suppose $f$ has a point $x \in C$ with $f^n(x) \in C$ for all $n \in \N$. 
Then $f^n(x) = F^n(x)$ for all $n \in \N$ and the forward orbit of $x$
has an accumulation point, $F^{n_k}(x) \to x_0 \in C$ as $k \to \infty$. 
A sufficiently small neighbourhood $U$ of $x_0$ then satisfies
$F(U) \cap U = \emptyset$ but clearly $F^{n_{k+1}-n_k}(U) \cap U \neq \emptyset$ for sufficiently large $k$, contradicting (\ref{eq:freedisk}). 
Hence such point never escaping $C$ does not exist. 
\end{proof}

\section{Non-self Maps for Compact Simply Connected Planar Sets}
\label{sec:gen_sets}

In this section we will generalize the results of the previous section to non-self maps of
compact, simply connected, locally connected sets in the real plane 
(also known as \textit{nonseparating Peano continua}). 
Note that in our terminology simply connected always implies connected. 

We denote the Riemann sphere by $\RS$. If $X \subset \RS$ is compact, connected and nonseparating, 
then its complement $U = \RS \setminus X$ is simply connected 
(an open set $U \subset \RS$ is simply connected if and only if both $U$ and $\RS \setminus U$ are connected). 

The classical Riemann mapping theorem states that if $U \subsetneq \C$ is non-empty, simply connected and open,
then there exists a biholomorphic map from $U$ onto the open unit disk $\D = \{z \in \C \colon \| z \| < 1 \}$, 
known as the \textit{Riemann map}. 
We will make use of the following stronger result by Carath\'eodory (see \cite[Theorem 17.14]{Milnor2006}).

\begin{exttheorem}[Carath\'eodory's Theorem]\label{thm:caratheodory}
 If $U \subsetneq \RS$ is non-empty, simply connected and open, 
$\RS \setminus U$ has at least two points, 
and additionally $\partial U$ (or $\RS \setminus U$) is locally connected, 
then the inverse of the Riemann map, $\phi \colon \D \to U$, extends continuously 
to a map from the closed unit disk $\overline \D$ onto $\overline U$.
\end{exttheorem}

As in the case of Jordan domains, we will need a requirement on the homeomorphisms to be orientation-preserving. 
To make sense of this notion for more general subsets of the plane, we cite the following result, 
simplified by our assumption of a 
nonseparating and locally connected set (see also Oversteegen and Tymchatyn \cite{Oversteegen2010}).

\begin{exttheorem}[Oversteegen and Valkenburg \cite{Oversteegen2011}]\label{thm:orpres}
 Let $X \subset \RS$ be compact, simply connected and locally connected, and $f \colon X \to Y \subset \RS$ a homeomorphism.
Then the following are equivalent:
\begin{enumerate}
 \item $f$ is isotopic to the identity on $X$;
 \item there exists an isotopy $F \colon \RS \times [0,1] \to \RS$ such that 
$F^0 = \operatorname{id}_\RS$ and $F^1\vert_X = f$;
 \item $f$ extends to an orientation-preserving homeomorphism of $\RS$;
 \item if $U = \RS \setminus X$ and $V = \RS \setminus Y$, 
then $f$ induces a homeomorphism $\hat f$ 
from the prime end circle of $U$ to the prime end circle of $V$ which preserves the circular order. 
\end{enumerate}
\end{exttheorem}

To explain assertion (4) in the above theorem, let $\phi \colon \overline \D \to \overline U$ 
and $\psi \colon \overline \D \to \overline V$ be the extended inverse Riemann maps given by Carath\'eodory's theorem, 
and denote $S^1 = \partial \D$. Then the homeomorphism $\hat f \colon S^1 \to S^1$ 
is said to be \textit{induced by $f$}, if 
\begin{equation*}
\psi\vert_{S^1} \circ \hat f = f\vert_{\partial U} \circ \phi\vert_{S^1}.
\end{equation*}
More details can be found in \cite{Oversteegen2011}, 
for an introduction to Carath\'eodory's theory of prime ends 
the reader is referred to Milnor's book \cite[Chapter 17]{Milnor2006}. 

Furthermore, we will make use of the following notation: 
For $\varepsilon > 0$, denote the closed $\varepsilon$-neighbourhood of a set $X \subset \R^2$ by
\begin{equation*}
 X_\varepsilon = \{ z \in \R^2 \colon \inf_{x \in X} \| x - z \| \leq \varepsilon\}. 
\end{equation*}
In the proof of our main theorem, we will consider $\varepsilon$-neighbourhoods of 
disconnected compact sets. A key fact we need is that any two given connected components of 
such a set are separated by the $\varepsilon$-neighbourhood of the set, if $\varepsilon > 0$
is chosen sufficiently small. This is the following technical lemma, whose proof can be found
in the appendix.
\begin{lemma}\label{lem:separation_lemma}
 Let $X$ be a compact subset of $\R^n, n \in \N$. 
Let $\mathcal{C} = \{ C_i \}_{i \in I}$ be the collection 
of connected components of $X$ and assume $|\mathcal{C}| \geq 2$.
Let $C, C' \in \mathcal{C}$ be two distinct connected components. 
Then for $\varepsilon > 0$ sufficiently small, $C$ and $C'$ 
lie in two distinct connected components of $X_\varepsilon$.
\end{lemma}

We are now ready to prove our main results. 

\begin{proof}[Proof of Theorem \ref{thm:main_thm}]
Let $C$ be a connected component of $X \cap Y$ which contains a periodic orbit of $f$. 
We will prove that $f$ has a fixed point in $C$.
The strategy of the proof is to construct Jordan-domain neighbourhoods of $X$ and $Y$ 
and to extend $f$ to a homeomorphism between these neighbourhoods, so that 
Theorem \ref{thm:jordan_general} can be applied to the extension. The existence 
of a fixed point for the extension will then be shown to imply the existence of a fixed 
point for $f$ in $C$. 

For fixed $\varepsilon > 0$, we first construct a closed neighbourhood 
$N^\varepsilon$ of $X$ with boundary $\Gamma = \partial N^\varepsilon$ 
a Jordan curve such that $\overline{N^\varepsilon \setminus X}$ 
has a foliation $\{ \gamma_z \}_{ z \in \Gamma}$ with the following properties: 
\begin{itemize}
\item each $\gamma_z$ is an arc connecting the point $z \in \Gamma$ with a point $x(z) \in \partial X$;
\item $\gamma_z \cap \partial X = \{x(z)\}$ and $\gamma_z \cap \Gamma = \{z\}$ for each $z \in \Gamma$;
\item $\bigcup_{z \in \Gamma} \gamma_z = \overline{N^\varepsilon \setminus X}$;
\item if $z \neq z'$, then either $\gamma_z \cap \gamma_{z'} = \emptyset$, 
or $x(z) = x(z')$ and $\gamma_z \cap \gamma_{z'} = \{x(z)\} \subset \partial X$;
\item each arc $\gamma_z$ lies within an $\varepsilon$-ball 
centered at its basepoint $x(z)\in \partial X$: $\gamma_z \subset B_\varepsilon(x(z))$.  
\end{itemize}
Observe that by the last property, $N^\varepsilon$ is included in the
$\varepsilon$-neighbourhood $X_\varepsilon$ of $X$.

We identify $\R^2 = \C$ and $\C \subset \RS$ in the usual way and denote $U = \RS \setminus X$. 
Then $U$ satisfies the hypotheses of Theorem \ref{thm:caratheodory}, so there exists
a continuous map $\phi \colon \overline \D \to \overline U$ (the extended inverse Riemann map),
 whose restriction to the open unit disk $\D$ is a conformal
homeomorphism from $\D$ to $U$ and $\phi(\partial \D)  = \partial U$. 
The map $\phi$ can be chosen such that $\phi(0) = \infty$.
Since $\overline \D$ is compact, $\phi$ is uniformly continuous, 
so we can find $\delta = \delta(\varepsilon) > 0$ such that for all $x,y \in \overline \D$ with 
$\| x - y \| < \delta$, $\|\phi(x) - \phi(y)\| < \varepsilon$.

Assume without loss of generality that $\delta < 1$ and set 
$A_\delta = \{ x \in \C \colon 1-\delta \leq \|x\| \leq 1 \}$ 
and $N^\varepsilon = \phi(A_\delta) \cup X$. 
Then $N^\varepsilon$ is a closed neighbourhood of $X$, whose boundary
$\Gamma = \partial N^\varepsilon = \phi(\{ x \in \D \colon \|x\| = 1-\delta \})$ is a Jordan curve.
We can then construct the foliation $\{\gamma_z \}_{ z \in \Gamma}$ 
of $\overline{N^\varepsilon \setminus X}$ 
by taking the images of radial lines in $A_\delta$ (see Fig.\ref{fig:foliation}):
\begin{equation*}
 \gamma_z = \phi(\{ r \cdot v \colon 1-\delta \leq r \leq 1 \}), \quad\text{ where } v = \frac{\phi^{-1}z}{\|\phi^{-1}z\|}. 
\end{equation*}

\begin{figure}
 \centering
\includegraphics[width=\textwidth]{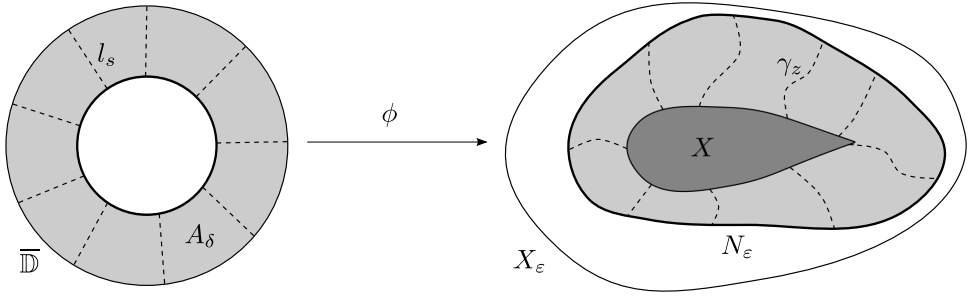}
 \caption{(Proof of Theorem \ref{thm:main_thm}) The extended inverse Riemann map $\phi \colon \overline{\D} \to \hat C \setminus \inter(X)$
maps the annulus $A_\delta = \{ x \in \C \colon 1-\delta \leq \|x\| \leq 1 \}$ onto $N^\varepsilon \setminus \inter(X)$
such that $\partial \D$ is mapped onto $\partial X$. 
Each leaf $\gamma_z$ of the foliation of $N^\varepsilon \setminus \inter(X)$ is the image under $\phi$ 
of a radial line segment $l_s \subset A_\delta$.}
 \label{fig:foliation}
\end{figure}

All the required properties of $N^\varepsilon$ and its foliation $\{\gamma_z\}$ then follow because 
$\phi$ maps $A_\delta \setminus \partial \D$ bijectively and uniformly continuously onto $N^\varepsilon \setminus X$; 
in particular the last required property follows by the choice of $\delta = \delta(\varepsilon)$.

Next, since $f$ is a homeomorphism, $Y = f(X)$ is simply connected, 
compact and also locally connected (see \cite[Theorem 17.15]{Milnor2006}). 
Hence, with $V = \RS \setminus Y$ and 
$\phi \colon \overline \D \to \overline V$ the corresponding 
extended inverse Riemann mapping, we can repeat the same construction as before 
to obtain a closed neighbourhood $M^\varepsilon$ of $Y$, such 
that $\Theta = \partial M^\varepsilon$ is a Jordan curve 
and $\{\theta_z \}$ is a foliation of 
$\overline{M^\varepsilon \setminus X}$ with the same properties as $\{\gamma_z\}$.

By choosing $\delta > 0$ small enough such that both 
constructions work with this given value, we get that
$N^\varepsilon = X \cup \phi(A_\delta)$ and $M^\varepsilon = Y \cup \psi(A_\delta)$. 
Denote by $\{l_s \colon s \in S^1\}$ the foliation of $A_\delta$ by radial lines such that 
$\gamma_z = \phi(l_s)$ for $z = \phi( (1-\delta)\cdot s)$ 
and $\theta_{z'} = \psi(l_s)$ for $z '= \psi( (1-\delta)\cdot s)$.

From Theorem \ref{thm:orpres} we get that $f$ induces a homeomorphism $\hat f \colon S^1 \to S^1$ such that
\begin{equation*}
\psi\vert_{S^1} \circ \hat f = f\vert_{\partial U} \circ \phi\vert_{S^1}.
\end{equation*}
Hence we can extend $f \colon X \to Y$ to $F_\varepsilon \colon N^\varepsilon \to M^\varepsilon$ by mapping each arc 
$\gamma_z \subset \overline{N^\varepsilon \setminus X}$ 
to the corresponding arc $\theta_{z'} \subset \overline{M^\varepsilon \setminus X}$.
More precisely, let $H\colon A_\delta \to A_\delta$ be the homeomorphism
which maps the radial line segment $l_s$ linearly to the radial line segment $l_{\hat f(s)}$
(i.e., in polar coordinates $H = \operatorname{id}\vert_{[1-\delta,1]} \times \hat f$).
Then define the homeomorphic extension of $f$ to $N^\varepsilon$ by setting
\begin{equation*}
 F_\varepsilon(x) = \begin{cases} f(x) &\mbox{if } x \in X, \\ 
\psi \circ H \circ \phi^{-1} (x) &\mbox{if } x \in N^\varepsilon\setminus X. \end{cases}
\end{equation*}
Note that for $0 < \varepsilon' < \varepsilon$, we can repeat the above construction to obtain 
Jordan domains $N^{\varepsilon'}$ and $M^{\varepsilon'}$ and 
a homeomorphism $F_{\varepsilon'} \colon N^{\varepsilon'} \to M^{\varepsilon'}$, 
and moreover $F_\varepsilon (x) = F_{\varepsilon'}(x)$ 
for all $x \in N^{\varepsilon'} \subset N^\varepsilon$.

Every connected component of $X \cap Y$ is a subset 
of a connected component of $N^\varepsilon \cap M^\varepsilon$. 
Let $K^\varepsilon$ be the connected component of 
$N^\varepsilon \cap M^\varepsilon$ which contains $C$. 

\begin{claim} \label{cl:fp}
 $F_{\varepsilon'}$ has a fixed point in $K^{\varepsilon'}$ 
for every $\varepsilon' \in (0, \varepsilon]$.
\end{claim}

This follows directly from the main result of the previous section: 
the periodic point for $f$ in $C$ is also a periodic point for 
$F_{\varepsilon'}$ in $K^{\varepsilon'}$. By Theorem \ref{thm:jordan_general}, 
$F_{\varepsilon'}$ has a fixed point in $K^{\varepsilon'}$.

\begin{claim} \label{cl:int_fp}
 $F_\varepsilon$ has a fixed point in $K^\varepsilon \cap (X \cap Y)$. 
\end{claim}

By Claim \ref{cl:fp} and the fact that for $0 < \varepsilon' < \varepsilon$, 
$F_{\varepsilon'}$ is the restriction of $F_\varepsilon$, $F_\varepsilon$ has 
a sequence of fixed points $(p_k)_{k \in \N}$ such that 
$p_k \in N^{\varepsilon/k} \cap M^{\varepsilon/k}$. By passing to a convergent subsequence, 
we get that $p_k \rightarrow p$ as $k \rightarrow \infty$, 
with $p \in K^\varepsilon \cap (X \cap Y)$ 
and $F_\varepsilon(p_k) = F_{\varepsilon/k}(p_k) = p_k$. 
By continuity $F_\varepsilon(p) = p$, which proves the claim. 

\begin{claim} \label{cl:good_fp}
 $F_\varepsilon$ has a fixed point in $C$. 
\end{claim}

By Claim \ref{cl:int_fp} we know that $F_\varepsilon$ 
(and hence $F_{\varepsilon'}$ for $\varepsilon' \in (0,\varepsilon]$) possesses a fixed point 
$p \in K^\varepsilon \cap (X \cap Y)$. If $p \notin C$, then $p$ lies
in another connected component of $K^\varepsilon \cap (X \cap Y)$, say $C'$. 
We can now apply Lemma \ref{lem:separation_lemma} to $K^\varepsilon \cap (X \cap Y)$
and its connected components. Since $N^{\varepsilon'} \subseteq X_{\varepsilon'}$ 
and $M^{\varepsilon'} \subseteq Y_{\varepsilon'}$,  
we obtain that for $\varepsilon' > 0$ sufficiently small, 
$p \notin K^{\varepsilon'}$. By again applying Theorem \ref{thm:jordan_general} 
and Claim \ref{cl:int_fp}, we get that $F_{\varepsilon'}$ has a fixed point 
$p' \in K^{\varepsilon'} \cap (X \cap Y)$, which is also a fixed point for $F_\varepsilon$. 

We can iterate this argument and obtain a sequence of $F_\varepsilon$-fixed points 
$(p_k)_{k \in \N}$ such that $p_k \in K^{1/k} \cap (X \cap Y)$.
As in the proof of the previous claim, by passing to a convergent subsequence,
$p_k \rightarrow p$ as $k \rightarrow \infty$, we get an $F_\varepsilon$-fixed point 
$p \in K^\varepsilon \cap (X \cap Y)$. 
If $p$ lies in a connected component of $K^\varepsilon \cap (X \cap Y)$ distinct from $C$, 
then, by Lemma \ref{lem:separation_lemma}, $p \notin K^{\varepsilon'}$ 
for sufficiently small $\varepsilon'$. But by construction, 
$p \in K^{\varepsilon'}$ for all $\varepsilon' \in (0, \varepsilon]$, so $p \in C$, finishing
the proof of Claim \ref{cl:good_fp}. 

\medskip

Claim \ref{cl:good_fp} implies that $f$ has a fixed point in $C$, 
and Theorem \ref{thm:main_thm} is proved. 
\end{proof}

\begin{proof}[Proof of Corollary \ref{cor:main_cor}]
If $C$ is a connected component of $X \cap Y$ which contains a point $x$ such that
$f^n(x) \in C$ for all $n \in \N$, then exactly as in the proof of Theorem \ref{thm:main_thm}, 
we can construct an extension $F$ of $f$ to a closed neighbourhood $N^\varepsilon$ of $X$, 
for small $\varepsilon > 0$. Using Corollary \ref{cor:jordan_general_rec}
instead of Theorem \ref{thm:jordan_general}, we get that the extension has a fixed point. 
The same argument then shows that one such fixed point lies in $C$, as required. 
\end{proof}

\section{Discussion of Assumptions and Possible Extensions}\label{sec:discussion}

Easy examples show that in our main results we cannot omit the hypothesis of a periodic orbit being completely contained 
in a given connected component of $X \cap f(X)$: 
take $X$ to be a closed $\varepsilon$-neighbourhood of the 
straight line segment $S = [-1,1]\times \{0\}$ and let $f \colon X \to Y$ 
be a homeomorphism which maps $S$ to the semicircle 
$\{(x,y) \colon x^2 + y^2 = 1, y \geq 0 \}$ with $f(-1,0) = (1,0)$ and $f(1,0) = (-1,0)$. 
For $\varepsilon > 0$ small, $f$ has a period-two orbit (spread over two different connected
components of $X \cap f(X)$) but no fixed point. 

Furthermore, obvious counterexamples show that our main result, Corollary \ref{cor:main_cor}, fails 
when the domain $X\subset \R^2$ of the homeomorphism $f$ is not compact. 
Also, we are not aware of generalizations to higher dimensions; in fact, simple counterexamples 
to Theorem \ref{thm:main_thm} in its current form can be constructed, 
when $X \subset \R^n$ is an $n$-dimensional ball with $n \geq 3$.

\medskip

Bonino \cite{Bonino2001} showed that if an orientation-preserving homeomorphism $f\colon \R^2 \to \R^2$ 
possesses non-escaping points in a topological disk $D$ bounded by a simple closed curve $C$, 
and $C$ can be split into two arcs $\alpha = [a,b]_C$, $\beta = [b,a]_C$ 
such that $D \cap f^{-1}(\beta) = \emptyset$ and $f^{-1}(D) \cap \alpha = \emptyset$ 
(Fig.\ref{fig:examples}(a)), then $f$ has a fixed point \textit{in $D$}.
Bonino's theorem is similarly topological, but does not imply (and is not implied by) our results.
\begin{figure}
\centering
\subfloat[][]{\includegraphics[width=0.4\textwidth]{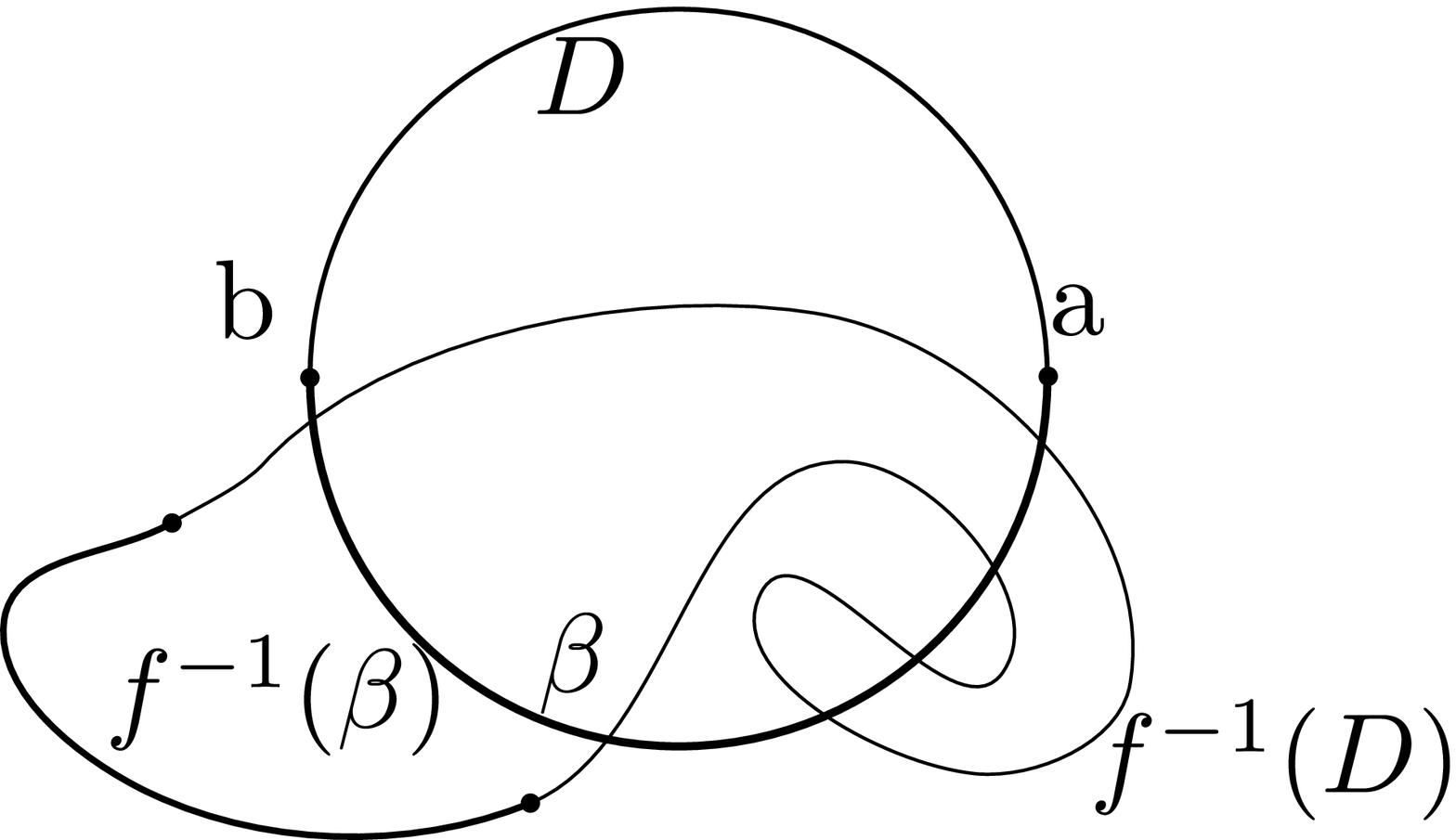}}
\qquad \qquad
\subfloat[][]{\includegraphics[width=0.4\textwidth]{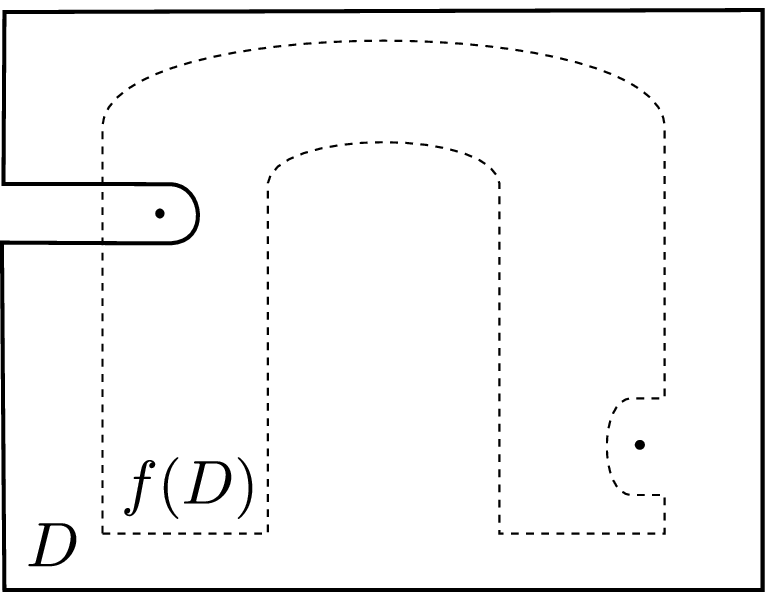}}
\caption{(a) Jordan domain satisfying Bonino's \cite{Bonino2001} condition.\newline
(b) Horseshoe map with neighbourhood of period-two point removed from its domain: The map has period-three points
but no period-two points, showing that a simple generalization of period-forcing results to
non-self maps of planar domains does not work.}
\label{fig:examples}
\end{figure}

Another result related to Theorem \ref{thm:main_thm} is due to Brown \cite[Theorem 5]{Brown1990}. 
Brown assumed that for the $n$-periodic point $x \in X$
there is a connected neighbourhood $W$ of $\{x\}\cup \{f(x)\}$ with $f^i(W) \subset X$ for $i = 1,\ldots,n$, and 
showed that $f$ then has a fixed point in $X$. 

\medskip

A somewhat different class of results relates to the question of period forcing.
Gambaudo et al \cite{Gambaudo1990} showed that a $C^1$ orientation-preserving embedding
of the disk into itself has periodic orbits of all periods, if it has a period-three orbit of 
braid type different than the rotation by angle $2\pi/3$ 
(see also Kolev \cite{Kolev1994} for a topological version of this result).
For $f(X) \nsubset X$, this statement is false: 
a counterexample is a version of the Smale horseshoe map, where a period-two point is removed from the disc
together with a narrow strip such that the remaining domain $X$ is still simply connected (see Fig.\ref{fig:examples}(b)). 
Then $f$, the restriction of the horseshoe map to $X$, 
is an orientation-preserving homeomorphism with $X \cap f(X)$ connected, and
one can choose the removed strip narrow enough such that $f$ still has period-three orbits 
but no period-two orbit in $X$. 
We conclude with the following question:
\begin{question}
Can one find conditions on $X$ and $f(X)$, under which \textbf{every extension} of $f$ 
to a homeomorphism of $\R^2$ has periodic orbits of every period \textbf{passing through} $X$,
whenever $f$ has a period-three point (of non-trivial braid type) in $X$?
\end{question}
The method of extending a non-self map of $X \subset \R^2$ to a self-map of $\R^2$ 
(or $\D$) without adding fixed (or periodic) points
could possibly also be applied to this problem.

\section*{Appendix. Proofs of Technical Lemmas}

\begin{proof}[Proof of Lemma \ref{lem:part}]
 We apply Proposition \ref{prop:jordan} with $U = D$ and $U' = E$.
The set $D \cap E$ is a Jordan domain and by (1), 
its boundary curve consists of the (pairwise disjoint) sets
\begin{itemize}
 \item $\partial D \cap \partial E$ (isolated points or closed arcs), 
\item $\alpha_i \subset \partial D$, $i \in I$ (open arcs), and
\item $\beta_j \subset \partial E$, $j \in J$ (open arcs).
\end{itemize}
It also follows from Proposition \ref{prop:jordan}, that $\overline{E \setminus D}$ is the 
union of Jordan domains $A_i$ bounded by $\alpha_i \cup [a_i,b_i]_{\partial E}$
and $\overline{D \setminus E}$ is the union of Jordan domains $B_j$ 
bounded by $\beta_j \cup [c_j,d_j]_{\partial D}$ (see Fig.\ref{fig:prop_jordan}).

By Schoenflies' theorem, each Jordan domain $A_i$
is homeomorphic to the closed unit disk $\overline \D = \{(x,y) \in \R^2 \colon x^2+y^2 \leq 1 \}$.
Let $h_i \colon A_i \to \overline \D$ be such 
homeomorphism. We can assume $h_i(a_i) = (-1,0)$ and $h_i(b_i) = (1,0)$. 
Then the partition of $\overline \D$ into vertical line segments 
$l_s \coloneqq \{(x,y) \in \overline \D \colon x = s \}$, $s \in (-1,1)$, 
gives rise to a partition $\{h_i^{-1}(l_s) \colon s \in (-1,1)\}$ of $A_i$ into closed arcs, 
each connecting a point of $\alpha_i$ to a point of $(a_i,b_i)_{\partial E}$. 
Similarly one can obtain a partition of $B_j$ into closed arcs 
connecting points on $\beta_j$ to points on $(c_j,d_j)_{\partial D}$. 

Note that $D \cup E$ is also a Jordan domain and that the 
$(a_i,b_i)_{\partial E}$, $i \in I$, and $(c_j,d_j)_{\partial D}$, $j \in J$,
together with $\partial D \cap \partial E$ form a partition of its boundary. 
Thus we obtained a collection of (pairwise disjoint) closed arcs, 
each connecting precisely one point of $\partial (D \cap E)$ to precisely one point of $\partial (D \cup E)$. 
Denote the arc corresponding to $z \in \partial (D \cup E)$ by $l^z$, 
and let $l^z = \{z\}$ whenever $z \in \partial D \cap \partial E$. 

Since $D \cup E$ is a Jordan domain, we can again apply Schoenflies' theorem to obtain a homeomorphism
$h \colon \R^2 \setminus \inter(D \cup E) \to \R^2 \setminus \D$. 
Let $r_\theta$ be the radial line segment in $\R^2 \setminus \D$ 
expressed in polar coordinates as $\{(r,\phi) \colon r \geq 1, \phi = \theta \}$. 
Then $\{r_\theta \colon \theta \in [0, 2\pi) \}$ forms a partition of $\R^2 \setminus \D$, 
which can be pulled back to a partition $\{h^{-1}(r_\theta) \colon\theta \in [0, 2\pi) \}$ of $\R^2 \setminus \inter(D \cup E)$. 
This partition consists of (pairwise disjoint) arcs $m^z$, each connecting a point on $z \in \partial (D \cup E)$ to $\infty$. 

Combining the above, we get that each point $z \in \partial (D \cup E)$ is the endpoint of 
two uniquely defined arcs $l^z$ and $m^z$ ($l^z$ possibly being the trivial arc $\{z\}$).
Let $L^z \coloneqq l^z \cup m^z$, which is an arc connecting a point on $\partial (D \cap E)$ to $\infty$, 
for every $z \in \partial (D \cup E)$. 
Then $\{L^z \colon z \in \partial (D \cup E) \}$ is a partition 
of $\R^2 \setminus \inter(D \cap E)$ with the desired properties.

At last, we construct a continuous function 
$\lambda^D \colon \R^2 \setminus \inter(D) \to \R_{\geq 0}$ strictly monotonically
increasing along each arc $L^z$ and such that $\lambda^D\vert_{\partial D} \equiv 0$; 
the function $\lambda^E \colon \R^2 \setminus \inter(D) \to \R_{\geq 0}$ can be constructed similarly. 

For $x \in \R^2 \setminus \inter(D)$, let $z$ be the unique point in $\partial (D \cup E)$ such that 
$x \in L^z$ and let $\{p\} = L^z \cap \partial D$. If $p \notin E$, then $p = z$ and the arc 
$[p,x] \subset L^z$ is contained in $\R^2 \setminus \inter(D \cup E)$. Then $h([p,x])$ is a straight line from
$h(p)$ to $h(x)$ and we set $\lambda^D(x) = \ell( h([p,x]) )$, where $\ell$ denotes the usual (Euclidean) arc length on arcs 
in $\R^2$. If $p \in E$ and $x \in E$, $[p,x] \subset L^z$ lies in a component $A_i$ of $\overline{E \setminus D}$, and 
we set $\lambda^D(x) = \ell( h_i([p,x]) )$. 
Finally, if $p \in E$ and $x \notin E$, $[p,x]$ is the union of the arcs $[p,z] \subset L^z$ and $[z,x] \subset L^z$, 
and we set $\lambda^D(x) = \ell( h_i([p,z]) ) + \ell( h([z,x])) $. 
\end{proof}

\begin{proof}[Proof of Lemma \ref{lem:separation_lemma}]
Suppose for a contradiction that $C$ and $C'$ lie in the same component $C_k$ of $X_{1/k}$ for all $k \in \N$. 
We use the fact that the space of subcontinua (non-empty, compact, connected subsets) of a compact set in the plane 
equipped with the Hausdorff metric is a compact complete metric space.
Then, by passing to a subsequence which we also call $C_k$, we have $C_k \to C_0$ as $k \to \infty$, 
where $C_0$ is itself non-empty, compact and connected. It is easy to see that convergence in the Hausdorff
metric implies that $C \cup C' \subseteq C_0 \subseteq X$, which contradicts the assumption that $C$ and $C'$ 
are distinct connected components of $X$.
\end{proof}

\section*{Acknowledgement}
I would like to thank my PhD advisor, Professor Sebastian van Strien, 
for the many invaluable discussions that led me to the results presented in this paper, 
and Dr.~Robbert Fokkink for suggesting to consider a generalization of the result from 
Jordan domains to more general planar sets. I also thank the anonymous referee for his 
careful reading of the manuscript and helpful suggestions improving the exposition of the paper.

\bibliographystyle{abbrv}
\bibliography{../BibTex/library}

\begin{thebibliography}{10}

\bibitem{Barge1993}
M.~Barge and J.~Franks.
\newblock {Recurrent sets for planar homeomorphisms}.
\newblock In {\em From Topology to Computation: Proceedings of the Smalefest
  (Berkeley, CA, 1990)}, pages 186--195, New York, 1993. Springer.

\bibitem{Bonino2001}
M.~Bonino.
\newblock {A dynamical property for planar homeomorphisms and an application to
  the problem of canonical position around an isolated fixed point}.
\newblock {\em Topology}, 40(6):1241--1257, 2001.

\bibitem{Bonino2002}
M.~Bonino.
\newblock {Lefschetz index for orientation reversing planar homeomorphisms}.
\newblock {\em Proc. Amer. Math. Soc.}, 130(7):2173--2177, 2002.

\bibitem{Brouwer1912}
L.~E.~J. Brouwer.
\newblock {Beweis des ebenen Translationssatzes}.
\newblock {\em Math. Ann.}, 72(1):37--54, 1912.

\bibitem{Brown1984}
M.~Brown.
\newblock {A new proof of Brouwer's lemma on translation arcs}.
\newblock {\em Houston J. Math.}, 10(1):35--41, 1984.

\bibitem{Brown1990}
M.~Brown.
\newblock {On the fixed point index of iterates of planar homeomorphisms}.
\newblock {\em Proc. Amer. Math. Soc.}, 108(4):1109--1114, 1990.

\bibitem{Fathi1987}
A.~Fathi.
\newblock {An orbit closing proof of Brouwer's lemma on translation arcs}.
\newblock {\em Enseign. Math.}, 33(3-4):315--322, 1987.

\bibitem{Franks2008}
J.~Franks.
\newblock {A new proof of the Brouwer plane translation theorem}.
\newblock {\em Ergodic Theory Dynam. Systems}, 12(02):217--226, 2008.

\bibitem{Gambaudo1990}
J.-M. Gambaudo, S.~van Strien, and C.~Tresser.
\newblock {Vers un ordre de Sarkovskiĭ pour les plongements du disque
  pr\'{e}servant l'orientation}.
\newblock {\em C. R. Math. Acad. Sci. Paris}, 310(5):291--294, 1990.

\bibitem{Kerekjarto1923}
B.~Ker\'{e}kj\'{a}rt\'{o}.
\newblock {\em {Vorlesungen \"{u}ber Topologie}}.
\newblock Springer, Berlin, 1923.

\bibitem{Kolev1994}
B.~Kolev.
\newblock {Periodic orbits of period 3 in the disc}.
\newblock {\em Nonlinearity}, 7:1067--1071, 1994.

\bibitem{LeCalvez1997}
P.~{Le Calvez} and J.-C. Yoccoz.
\newblock {Un th\'{e}or\`{e}me d'indice pour les hom\'{e}omorphismes du plan au
  voisinage d'un point fixe}.
\newblock {\em Ann. of Math.}, 146(2):241--293, 1997.

\bibitem{Milnor2006}
J.~Milnor.
\newblock {\em {Dynamics in One Complex Variable}}.
\newblock Princeton University Press, 3rd edition, 2006.

\bibitem{Oversteegen2010}
L.~G. Oversteegen and E.~D. Tymchatyn.
\newblock {Extending isotopies of planar continua}.
\newblock {\em Ann. of Math.}, 172(3):2105--2133, 2010.

\bibitem{Oversteegen2011}
L.~G. Oversteegen and K.~I.~S. Valkenburg.
\newblock {Characterizing isotopic continua in the sphere}.
\newblock {\em Proc. Amer. Math. Soc.}, 139(04):1495--1495, 2011.

\end{thebibliography}

\end{document}